\documentclass[final,11pt]{elsarticle}
\usepackage{amsmath}
\usepackage{amsthm}
\usepackage{amssymb}
\usepackage{mathtools}
\usepackage{array}
\usepackage{booktabs}
\usepackage{bm}
\usepackage{tikz}
\usepackage{xcolor}

\usepackage{enumerate}
\usepackage{enumitem}
\setlist[enumerate,1]{label=(\arabic*),font=\textup,
leftmargin=7mm,labelsep=1.5mm,topsep=0mm,itemsep=-0.8mm}
\setlist[enumerate,2]{label=(\alph*),font=\textup,
leftmargin=7mm,labelsep=1.5mm,topsep=-0.8mm,itemsep=-0.8mm}

\usepackage{geometry}
\usepackage{microtype}

\usepackage[pdfstartview=FitH,bookmarksopen=false,
colorlinks=true,linkcolor=blue,citecolor=blue]{hyperref}

\newtheorem{theorem}{Theorem}[section]

\newtheorem{lemma}{Lemma}[section]

\newtheorem{conjecture}{Conjecture}[section]

\theoremstyle{definition}

\numberwithin{equation}{section}

\begin{document}
	\begin{frontmatter}
		\title{Some results on the Turán number of $k_1P_{\ell}\cup k_2S_{\ell-1}$\,\tnoteref{titlenote}}
		\tnotetext[titlenote]{This work was supported by the National Nature Science Foundation of China (Nos.11871040, 12271337)}
		
		\author{Tao Fang}
		\ead{tao2021@shu.edu.cn}
		
		\author{Xiying Yuan\corref{correspondingauthor}}
        \cortext[correspondingauthor]{Corresponding author. Email address: {\tt xiyingyuan@shu.edu.cn} (Xiying Yuan)}

		\address{Department of Mathematics, Shanghai University, Shanghai 200444, P.R. China}
				
		\begin{abstract}
		  The Turán number of a graph $H$, denoted by $ex(n, H)$, is the maximum number of edges in any graph on $n$ vertices containing no $H$ as a subgraph. Let $P_{\ell}$ denote the path on $\ell$ vertices, $S_{\ell-1}$ denote the star on $\ell$ vertices and $k_1P_{\ell}\cup k_2S_{\ell-1}$ denote the path-star forest with disjoint union of $k_1$ copies of $P_{\ell}$ and $k_2$ copies of $S_{\ell-1}$. In 2013, Lidický et al. first considered the Turán number of $k_1P_4\cup k_2S_3$ for sufficiently large $n$. In 2022, Zhang and Wang raised a conjecture about the Turán number of $k_1P_{2\ell}\cup k_2S_{2\ell-1}$. In this paper, we determine the Turán numbers of $P_{\ell}\cup kS_{\ell-1}$, $k_1P_{2\ell}\cup k_2S_{2\ell-1}$, $2P_5\cup kS_4$ for $n$ appropriately large, which implies the conjecture of Zhang and Wang. The corresponding extremal graphs are also completely characterized.
		\end{abstract}
		
	\end{frontmatter}
	
\section{Introduction}\label{sec1}
In this paper, all graphs considered are undirected, finite and contain neither loops nor multiple edges. The vertex set of a graph $G$ is denoted by $V(G)$, the edge set of $G$ by $E(G)$, the number of the vertices in $G$ by $v(G)$ and the number of edges in $G$ by $e(G)$. Let $K_n$, $P_n$, $S_{n-1}$ denote the complete graph, path and star on $n$ vertices, respectively. For a vertex $v\in V(G)$, let $N_G(v)$ denote the set of vertices in $G$ which are adjacent to $v$ and $d_G(v)$ denote the degree of a vertex $v$, i.e., $d_G(v)=\left|N_G(v)\right|$. Given two vertex-disjoint graphs $G$ and $H$, let $G\cup H$ denote the disjoint union of graphs $G$ and $H$, $kG$ the disjoint union of $k$ copies of $G$, and $G\vee H$ the graph obtained from $G\cup H$ by joining all vertices of $G$ to all vertices of $H$. We use $\overline{G}$ to denote the complement of the graph $G$. For any set $S\subseteq V(G)$, let $G[S]$ denote the subgraph of $G$ induced by $S$, $\left|S\right|$ denote the cardinality of $S$. For a graph $G$ and its subgraph $H$, let $G-H$ denote the subgraph induced by $V(G)\backslash V(H)$.

The Turán number of a graph $H$, $ex(n, H)$, is the maximum number of edges in $G$ of order $n$ that does not contain a copy of $H$. Denote by $\mathbb{EX}(n,H)$ the set of graphs on $n$ vertices with $ex(n,H)$ edges containing no $H$ as a subgraph and call the graph from $\mathbb{EX}(n, H)$ the extremal graph for $H$ or $H$-extremal graph. If $\mathbb{EX}(n, H)$ contains only one graph, we may simply use $\mbox{\rm{EX}}(n, H)$ instead.

The study of Turán numbers of forests began with the famous result of Erdös and Gallai \cite{Erdos1959}  in 1956. Then in 1975, Faudree and Schelp \cite{faudree1975path} gave an improvement of the extremal graph for $P_k$.

\begin{theorem}\cite{Erdos1959}\label{Erdos}
  Let $n=d(\ell-1)\geqslant 2$, where $d\geqslant 1$. Then $$ex(n, P_{\ell})= \frac{(\ell-2)n}{2}.$$
  Furthermore, $$\mbox{\rm{EX}}(n, P_{\ell})=dK_{\ell-1}.$$
\end{theorem}

The following two symbols are defined in \cite{yuan2021turan}. Let $n\geqslant m\geqslant \ell \geqslant 2$ be three positive integers and $n=(m-1)+d(\ell-1)+r$ with $d\geqslant 0$ and $0\leqslant r<\ell -1$. Define
\begin{equation*}
[n, m, \ell ]=\binom{m-1}{2}+d\binom{\ell-1}{2}+\binom{r}{2}.
\end{equation*}
Let $n$ and $s$ be two positive integers and $n\geqslant s$. Define
\begin{equation*}
[n,s]=\binom{s-1}{2}+(s-1)(n-s+1).
\end{equation*}

\begin{theorem}\cite{faudree1975path}\label{faudree}
   Let $n=d(\ell-1)+r$, where $d\geqslant 1$ and $0\leqslant r<\ell-1$. Then $$ex(n, P_{\ell})= [n, \ell, \ell].$$
   Furthermore, if $\ell$ is even, $r={\ell}/2$ or $(\ell-2)/2$, then
   $$\mathbb{EX}(n, P_{\ell})=\left\{dK_{\ell-1}\cup K_r, \left(\left(d-s-1\right)K_{\ell-1}\right)\cup \left(K_{\frac{\ell-2}{2}}\vee \overline{K}_{\frac{\ell}{2}+s\left(\ell-1\right)+r}\right), s=0, 1, \cdots, d-1\right\};$$
   if otherwise, then
   $$\mbox{\rm{EX}}(n, P_{\ell})=dK_{\ell-1}\cup K_r.$$
\end{theorem}

We follow the notation and terminology of \cite{liu2013turan}. A linear forest is a forest whose connected components are paths. A star forest is a forest whose connected components are stars. A path-star forest is a forest whose connected components are paths and stars. In 2011, Bushaw and Kettle \cite{bushaw2011turan} determined the Turán numbers of $kP_{\ell}$ for sufficiently large $n$, which was extended by Lidiciký et al. \cite{liu2013turan}. Yuan and Zhang \cite{yuan2017turan, yuan2021turan} determined the Turán numbers of linear forests containing at most one odd path for all $n$. For special linear forest, Bielak and Kieliszek \cite{bielak2016} and Yuan and Zhang \cite{yuan2021turan} independently determined $ex(n, 2P_5)$ for all $n$ and characterized all extremal graphs.

\begin{lemma}\cite{bielak2016,yuan2021turan}\label{2P5}
  Let $n\geqslant10$. Then
  $$ex(n, 2P_5)=\mbox{\rm{max}}\big\{[n, 10,5], 3n-5\big\}.$$
  The extremal graphs are $K_9\cup \mbox{\rm{EX}}(n, P_5)$ and $K_3\vee\left(K_2\cup \overline{K}_{n-5}\right)$.
\end{lemma}

By calculations, when $n\geqslant 38$, $[n, 10, 5]<3n-5$ holds. Hence, we may get the following result from Lemma \ref{2P5}.

\begin{lemma}\label{2P5g}
  When $n\geqslant38$, we have
  $$ex(n, 2P_5)=3n-5.$$
  The extremal graph is $K_3\vee\left(K_2\cup \overline{K}_{n-5}\right)$.
\end{lemma}

The following lemma is based on Theorem 1.7 of \cite{yuan2021turan}.

\begin{lemma} \cite{yuan2021turan}\label{cor}
  Let $k\geqslant 2$ be a positive integer, $\ell$ be an even number and $n\geqslant \ell k$. Then
    $$ex(n, kP_{\ell})=\mbox{\rm{max}$\big\{[n, \ell k, \ell], [n, \ell k/2 ]\big\}$}.$$
  The extremal graphs are $\mbox{\rm{EX}}(n-\ell k+1, P_{\ell})\cup K_{k \ell -1}$ and $K_{\ell k/2-1}\vee \overline{K}_{n-\ell k/2+1}$.
\end{lemma}

By calculations, when $k\geqslant 2$ and $n\geqslant (2{\ell}^2+3\ell-4)k+3$, $[n, \ell k, \ell]<[n, \ell k/2]$ holds. Hence, we may get the following result from Lemma \ref{cor}.

\begin{lemma}\label{ourlem}
  Suppose $k\geqslant 2$, $\ell$ are positive integers and $n\geqslant (2{\ell}^2+3\ell-4)k+3$. Then
  $$ex(n, kP_{2\ell})=\binom{\ell k-1}{2}+(\ell k-1)(n-\ell k+1).$$
  The extremal graph is $K_{\ell k-1}\vee \overline{K}_{n-\ell k+1}$.
\end{lemma}

For sufficiently large $n$, Lidický et al. \cite{liu2013turan} determined the Turán number of stars forests. Later, Lan et al. \cite{lan2019turan} determined the Turán number of $kS_{\ell}$ for $n$ appropriately large related to $k$ and $\ell$. Furthermore, Li et al. \cite{li2022turan} determined the Turán number of $kS_{\ell}$, where $k\geqslant 2$ and $\ell \geqslant 3$, for all $n$.

\begin{lemma}\cite{lan2019turan} \label{sl}
  If $\ell\geqslant3$ and $n\geqslant \ell+1$, then
  $$ex(n, S_{\ell})\leqslant\left\lfloor\frac{(\ell-1)n}{2}\right\rfloor,$$
  with one extremal graph is the $\left(\ell-1\right)$-regular graph on $n$ vertices.
\end{lemma}

\begin{theorem}\cite{li2022turan}\label{li2022}
  If $k\geqslant 2$ and $\ell \geqslant 3$, then
  \begin{equation*}
    ex(n, kS_{\ell})=\begin{cases}
                  \binom{n}{2}, & \text{if } n<k(\ell+1), \\
                  \binom{k\ell+k-1}{2}+\binom{n-k\ell-k+1}{2}, & \text{if }k(\ell+1)\leqslant n\leqslant(k+1)\ell+k-1,  \\
                  \binom{k\ell+k-1}{2}+\left\lfloor\frac{(\ell-1)(n-k\ell-k+1)}{2} \right\rfloor, & \text{if }(k+1)\ell+k\leqslant n<\frac{k{\ell}^2+2k\ell+2k-2}{2},  \\
                  \binom{k-1}{2}+(n-k+1)(k-1)+\left\lfloor\frac{(\ell-1)(n-k+1)}{2} \right\rfloor, & \text{if }n\geqslant \frac{k{\ell}^2+2k\ell+2k-2}{2} .
                \end{cases}
  \end{equation*}
\end{theorem}

In this paper, we mainly consider the Turán numbers of some kinds of path-star forests. The Turán numbers and the extremal graphs for $P_{\ell}\cup kS_{\ell-1}$, $k_1P_{2\ell}\cup k_2S_{2\ell-1}$ and $2P_5\cup kS_4$ will be presented in Section \ref{sec2}, and their proofs will be provided in Section \ref{sec3}.

\section{Main results}\label{sec2}

Now, we introduce the following three kinds of graphs to state the main results. Set
$$G_1(n, k, \ell)=K_{k}\vee (dK_{\ell-1}\cup K_r), \mbox{ where $n=k+d(\ell-1)+r$, $0\leqslant r<\ell-1$,}$$
$$G_2(n, k_1, k_2, 2\ell)=K_{\ell k_1+k_2-1}\vee \overline{K}_{n-\ell k_1-k_2+1},$$
$$G_3(n, k)=K_{k+3}\vee\left(K_2\cup \overline{K}_{n-k-5}\right).$$
By calculations, we have the following facts.
\begin{equation}\label{11}
  e(G_1(n, k, \ell))=\left(k+\frac{\ell}{2}-1\right)n-\frac{k^2+(\ell-1)(k+r)-r^2}{2},
\end{equation}
\begin{equation}\label{2}
  e(G_2(n, k_1, k_2, 2\ell))=(\ell k_1+k_2-1)n-\frac{(\ell k_1+k_2)(\ell k_1+k_2-1)}{2},
\end{equation}
\begin{equation}\label{3}
  e(G_3(n, k))=(k+3)n-\frac{k^2+7k+10}{2}.
\end{equation}

Denote a kind of path-star forest by $F(k_1, k_2; \ell)=k_1P_{\ell}\cup k_2S_{\ell-1}$. Lidický et al. \cite{liu2013turan} first investigated the Turán number of $F(k_1, k_2; 4)$ for sufficiently large $n$. Lan et al. \cite{lan2019turan} considered the Turán number of $F(k_1, k_2; 4)$ for $n\geqslant 10k_1+13k_2+3$. Later, Zhang and Wang \cite{zhang2022turan} considered the Turán number of $F(k_1, k_2; 6)$ for $n\geqslant 23k_1+31k_2+3$ and proposed Conjecture \ref{conj2}.

\begin{theorem}\cite{lan2019turan}\label{lan2019}
 Suppose $n=k_2+3d+r\geqslant 10k_1+13k_2+3$, where $k_1, k_2, d, r$ are positive integers and $r\leqslant 2$. Then
 $$ex(n, F(k_1, k_2; 4))=\mbox{\rm{max}}\big\{e(G_1(n, k_2, 4)), e(G_2(n, k_1, k_2, 4))\big\}.$$
 Furthermore, the extremal graph is $G_1(n, k_2, 4)$ when $k_1=1$ and $G_2(n, k_1, k_2, 4)$ when $k_1>1$. In particular, $G_2(n, k_1, k_2, 4)$ is also an extremal graph when $k_1=1$ and $r=1$ or $r=2$.
\end{theorem}

\begin{theorem}\cite{zhang2022turan} \label{zhang2022}
Suppose $n=k_2+5d+r\geqslant 23k_1+31k_2+3$, where $k_1, k_2, d, r$ are positive integers and $r\leqslant 4$. Then
$$ex(n, F(k_1, k_2; 6))=\mbox{\rm{max}}\big\{e(G_1(n, k_2, 6)), e(G_2(n, k_1, k_2, 6))\big\}.$$
Furthermore, the extremal graph is $G_1(n, k_2, 6)$ when $k_1=1$ and $G_2(n, k_1, k_2, 6)$ when $k_1>1$.
\end{theorem}

\begin{conjecture}\cite{zhang2022turan}\label{conj2}
  Suppose $k_1\geqslant 1$, $k_2$ and $\ell \geqslant 2$  are integers and $n=k_2+d(2\ell -1)+r$, where $0\leqslant r< 2\ell-1$. Then
  $$ex(n, F(k_1, k_2; 2\ell))=\mbox{\rm{max}$\big\{e(G_1(n, k_2, 2\ell)), e (G_2(n, k_1, k_2, 2\ell))\big\}$}.$$
\end{conjecture}

We may point out that when $k_1=1$ and $r=2$ or $r=3$, $G_2(n, k_1, k_2, 6)$ is also an extremal graph of $F(k_1, k_2; 6)$. This fact was ignored in Theorem \ref{zhang2022}. Our results are given in the next three theorems, which determine the Turán numbers and the extremal graphs for $F(1, k; \ell)$ (see Theorem \ref{mainthm1}), $F(k_1, k_2; 2\ell)$ (see Theorem \ref{mainthm2}) and $F(2, k; 5)$ (see Theorem \ref{mainthm3}), respectively. The results of Theorem \ref{mainthm1} and Theorem \ref{mainthm2} imply Conjecture \ref{conj2}.

\begin{theorem}\label{mainthm1}
 Suppose $n =k+d(\ell-1)+r\geqslant (\ell^2-\ell+1)k+(\ell^2+3\ell-2)/2$,
 where $\ell\geqslant 4, 0\leqslant
 r< \ell-1$. Then
 \begin{equation*}
   ex(n, F(1, k; \ell))=e(G_1(n, k, \ell)).
 \end{equation*}
Moreover,
\begin{equation*}
\mathbb{EX} (n, F(1, k; \ell))=
\begin{cases}
  \left\{G_1(n, k, \ell), G_2(n, 1, k, \ell)\right\}, & \mbox{if $\ell$ is even, and $r=\frac{\ell}{2}$ or $r=\frac{\ell-2}{2}$,} \\
  \left\{G_1(n, k, \ell)\right\}, & \mbox{otherwise}.
\end{cases}
\end{equation*}
\end{theorem}

\begin{theorem}\label{mainthm2}
 Suppose $n \geqslant (2{\ell}^2+3\ell-4)k_1+(4{\ell}^2-2\ell+1)k_2+3$,
 where $k_1\geqslant 2, \ell \geqslant 2$. Then
 \begin{equation*}
   ex(n, F(k_1, k_2; 2\ell))= e (G_2(n, k_1, k_2, 2\ell)).
 \end{equation*}
Moreover,
\begin{equation*}
\mbox{\rm{EX}} (n, F(k_1, k_2; 2\ell))=G_2(n, k_1, k_2, 2\ell).
\end{equation*}
\end{theorem}

\begin{theorem}\label{mainthm3}
 Suppose $n \geqslant 21k+38$. Then
 \begin{equation*}
   ex(n, F(2, k; 5))= e(G_3(n, k)).
 \end{equation*}
Moreover,
\begin{equation*}
\mbox{\rm{EX}} (n, F(2, k; 5))=G_3(n, k).
\end{equation*}
\end{theorem}

\section{Proofs of the main results}\label{sec3}

\subsection{The Turán number and the extremal graphs for $F(1, k; \ell)$}

Write $n=k+d(\ell-1)+r$, where $0\leqslant r< \ell-1$, and $$H=K_k\vee \left(\left(\left(d-s-1\right)K_{\ell-1}\right)\cup \left(K_{\frac{\ell-2}{2}}\vee \overline{K}_{\frac{\ell}{2}+s\left(\ell-1\right)+r}\right)\right),$$
where $\ell$ is an even integer.
Recall that $F(1, k; \ell)=P_{\ell}\cup kS_{\ell-1}$.
We first present the following lemma which help us to determine the extremal graphs for $F(1, k; \ell)$.

\begin{lemma}\label{d-1}
If $n\geqslant \ell k+\ell$ and $s\in\{0, 1, \cdots, d-2\}$, then $H$ contains a copy of $F(1, k; \ell)$.
\end{lemma}
\begin{proof}
If $s\in\{0, 1, \cdots, d-2\}$, we have $d-s-1\geqslant1$. In $H$, let
$$V(K_k)=\{u_1, u_2, \cdots, u_k\},
$$ $$V\left((d-s-1)K_{\ell-1}\right)=\{v_1, v_2, \cdots, v_{\ell-1}\}\cup V_1,$$
$$V\left(K_{\frac{\ell-2}{2}}\vee \overline{K}_{\frac{\ell}{2}+s\left(\ell-1\right)+r}\right)=\{w_1, w_2, \cdots, w_{\ell}\}\cup V_2,$$
where $v_1, v_2, \cdots, v_{\ell-1}$ are the vertices of an induced subgraph $K_{\ell-1}$ of $(d-s-1)K_{\ell-1}$, and $w_1\in V\left(K_{\frac{\ell-2}{2}}\right)$.
We may check that $H[\{u_1, v_1, v_2, \cdots, v_{\ell-1}\}]$ is a path on $\ell$ vertices, and $H[\{w_1, w_2, \cdots, w_{\ell}\}]$ is a star on $\ell$ vertices with center vertex $w_1$. We may find another $(k-1)$ copies of $S_{\ell-1}$ with $(k-1)$ center vertices in $\{u_2, u_3, \cdots, u_k\}$ and $\left(k-1\right)\left(\ell-1\right)$ leaves vertices in $V_1\cup V_2$. Hence, we have $F(1, k; \ell)\subseteq H$.
\end{proof}

\begin{proof}[Proof of Theorem \ref{mainthm1}]

We suppose $n\geqslant (\ell^2-\ell+1)k+(\ell^2+3\ell-2)/2$ in this subsection. Recall that $$G_1(n, k, \ell)=K_{k}\vee \left(dK_{\ell-1}\cup K_r\right)$$
and
$$G_2(n, 1, k, \ell)=K_{k+\frac{\ell}{2}-1}\vee \overline{K}_{n-k-\frac{\ell}{2}+1}.$$

First we prove that both $G_1(n, k, \ell)$ and $G_2(n, 1, k, \ell)$ are $F(1, k; \ell)$-free. If $G_1(n, k, \ell)$ contains a copy of $F(1, k; \ell)$, then each $S_{\ell-1}$ contains at least one vertex of $K_k$, and the $P_{\ell}$ contains at least one vertex of $K_k$, which is a contradiction.  If $\ell$ is even and $G_2(n, 1, k, \ell)$ contains a copy of $F(1, k; \ell)$, then each $S_{\ell-1}$ contains at least one vertex of $K_{k+\frac{\ell}{2}-1}$, and the $P_{\ell}$ contains at least $\ell/2$ vertices of $K_{k+\frac{\ell}{2}-1}$, which is a contradiction. Hence, $G_2(n, 1, k, \ell)$ is $F(1, k; \ell)$-free. Furthermore, by (\ref{11}) and (\ref{2}), we have
\begin{equation*}
  e(G_1(n, k, \ell))-e(G_2(n, 1, k, \ell))=\frac{(\ell-2r)(\ell-2r-2)}{8}\geqslant0,
\end{equation*}
with equality if and only if $r=\ell/2$ or $r=(\ell-2)/2$. Thus we have
\begin{equation}\label{eq1>=}
  ex(n, F(1, k; \ell))\geqslant \mbox{\rm{max}}\big\{e(G_1(n, k, \ell)), e(G_2(n, 1, k, \ell))\big\}=e (G_1(n, k, \ell)).
\end{equation}

Now we will show the inequality
\begin{equation}\label{eq1<=}
  ex(n, F(1, k; \ell))\leqslant e(G_1(n, k, \ell))
\end{equation}
by induction on $k$. For $k=0$, $n=d(\ell-1)+r$, we have $G_1(n, 0, \ell)=dK_{\ell-1}\cup K_r$ and $F(1, 0; \ell)=P_{\ell}$. Then by Theorem \ref{faudree}, $ex(n, F(1, 0; \ell))=[n, \ell, \ell]=e(G_1(n, 0, \ell))$ holds. Suppose that $k\geqslant1$ and (\ref{eq1<=}) holds for all $k'<k$.
Let $G$ be an $n$-vertex $F(1, k; \ell)$-free graph with $e(G)=ex(n, F(1, k; \ell))$. By (\ref{eq1>=}) and (\ref{11}), we have
\begin{align*}
  e(G) & \geqslant  e(G_1(n, k, \ell)) \\
  & = \left(k+\frac{\ell}{2}-1\right)n-\frac{k^2+(\ell-1)(k+r)-r^2}{2} \\
  & > \left(k+\frac{\ell}{2}-2\right)n-\frac{(k-1)(k+\ell-2)}{2} \\
  & = (k-1)\left(n-\frac{k}{2}\right)+\frac{(\ell-2)(n-k+1)}{2} \\
  & \geqslant (k-1)\left(n-\frac{k}{2}\right)+\left\lfloor\frac{(\ell-2)(n-k+1)}{2}\right\rfloor \\
  & \geqslant ex(n,kS_{\ell-1}),
\end{align*}
which implies $G$ contains $k$ copies $S_{\ell-1}$ by Theorem \ref{li2022} and Lemma \ref{sl}. By induction hypothesis,
$$ex(n-\ell, F(1, k-1; \ell)) \leqslant e(G_1(n-\ell, k-1, \ell)).$$
Since $G$ is $F(1, k; \ell)$-free, $G-S_{\ell-1}$ is $F(1, k-1; \ell)$-free. Hence,
\begin{equation}\label{G-S}
  e (G-S_{\ell-1})\leqslant ex(n-\ell, F(1, k-1; \ell))\leqslant e (G_1(n-\ell, k-1, \ell)).
\end{equation}
Let $m_0$ be the number of edges incident with the vertices of $S_{\ell-1}$ in $G$, that is $m_0=e(G)-e(G-S_{\ell-1})$. Noting that $n\geqslant (\ell^2-\ell+1)k+(\ell^2+3\ell-2)/2$, by (\ref{eq1>=}) and (\ref{G-S}), we have
\begin{align*}
  m_0&=e(G)-e(G-S_{\ell-1}) \\
     &\geqslant e(G_1(n, k, \ell))-e(G_1(n-\ell, k-1, \ell)) \\
     &=n+(\ell-1)k+\frac{\ell^2-5\ell+2}{2} \\
     &\geqslant \ell(\ell k+\ell-1).
\end{align*}
That is, each copy of $S_{\ell-1}$ in $G$ contains a vertex with degree at least $\ell k+\ell-1$. Let $U\subseteq V(G)$ be a set of vertices with degree at least $\ell k+\ell-1$ and each vertex in $U$ belongs to distinct $S_{\ell-1}$. Then $\left|U\right|=k$. Let $\overline{U}=V(G)\backslash U$. Then $\left|\overline{U}\right|=n-k$.
Set $N\left(U\right)=\bigcup_{u\in U}N\left(u\right)$ and $W_0=N\left(U\right)\cap \overline{U}$. Then $\left|W_0\right|\geqslant (\ell-1)k+\ell$.
If $G[\overline{U}]$ contains a copy of $P_{\ell}$,  we set $W_1=W_0\backslash V(P_{\ell})$, then we have
$\left|W_1\right| \geqslant \left|W_0\right|-\ell \geqslant \left(\ell-1\right)k$.
For any $u\in U$, we have
\begin{equation*}
  d_{G[W_1]}\left(u\right) \geqslant (\ell k+\ell-1)-(k-1)-\ell=(\ell-1)k.
\end{equation*}
We may find $k$ copies of $S_{\ell-1}$ in $G-P_{\ell}$ with $k$ center vertices in $U$ and $(\ell-1)k$ leaves vertices in $W_1$. Hence $F(1, k; \ell) \subseteq G$, which is a contradiction.
Therefore, $G[\overline{U}]$ is $P_{\ell}$-free.
Recall that $\left|\overline{U}\right|=n-k=d(\ell-1)+r$. By Theorem \ref{faudree}, we have
\begin{equation}\label{<=}
  e\left(G[\overline{U}]\right)\leqslant ex(n-k, P_{\ell})=[n-k, \ell, \ell].
\end{equation}
Hence, by (\ref{<=}) and (\ref{11}), we have
\begin{align*}
  e(G) & =e(G[U])+e\left(U, \overline{U}\right)+e\left(G[\overline{U}]\right) \\
   & \leqslant \binom{k}{2}+k(n-k)+[n-k, \ell, \ell] \\
   & = \left(k+\frac{\ell}{2}-1\right)n-\frac{k^2+(\ell-1)(k+r)-r^2}{2} \\
   & = e(G_1(n, k, \ell)).
\end{align*}
Thus (\ref{eq1<=}) holds, and therefore, $e(G)=ex(n, F(1, k; \ell))=e(G_1(n, k, \ell))$ holds.

Now we determine the extremal graphs for $F(1, k; \ell)$. If $e(G)=e(G_1(n, k, \ell))$, then the equality case of (\ref{<=}) holds, and $G=K_k\vee G[\overline{U}]$. By Theorem \ref{faudree}, we consider the following two cases. (a) $G[\overline{U}]=dK_{\ell-1}\cup K_r$, where $0\leqslant r<\ell-1$. Then $$G=K_k\vee\left(dK_{\ell-1}\cup K_r\right)=G_1(n, k, \ell).$$
(b) $\ell$ is even, $r={\ell}/2$ or $r=(\ell-2)/2$, and $G[\overline{U}]=\left(\left(d-s-1\right)K_{\ell-1}\right)\cup \left(K_{\frac{\ell-2}{2}}\vee \overline{K}_{\frac{\ell}{2}+s\left(\ell-1\right)+r}\right)$,
where $s=0, 1, \cdots, d-1$. Noting that $G$ is $F(1, k; \ell)$-free, we have $s=d-1$ by Lemma \ref{d-1}, and then $G=K_{k+\frac{\ell}{2}-1}\vee \overline{K}_{n-k-\frac{\ell}{2}+1}=G_2(n, 1, k,\ell)$.

Hence, the extremal graph for $F(1, k; \ell)$ is $G_1(n, k, \ell)$, or $G_2(n, 1, k, \ell)$ if $\ell$ is even, $r=\ell/2$ or $r=(\ell-2)/2$.
The proof is completed.
\end{proof}

\subsection{The Turán number and the extremal graph for $F(k_1, k_2; 2\ell)$}

\begin{proof}[Proof of Theorem \ref{mainthm2}]

We suppose $n\geqslant (2{\ell}^2+3\ell-4)k_1+(4{\ell}^2-2\ell+1)k_2+3$ in this subsection. Recall that $$G_2(n, k_1, k_2, 2\ell)=K_{\ell k_1+k_2-1}\vee \overline{K}_{n-\ell k_1-k_2+1}$$
and
$$F(k_1, k_2; 2\ell)=k_1P_{2\ell}\cup k_2S_{2\ell-1}.$$

If $G_2(n, k_1, k_2, 2\ell)$ contains a copy of $F(k_1, k_2; 2\ell)$, then each $S_{2\ell-1}$ contains at least one vertex of $K_{\ell k_1+k_2-1}$ and each $P_{2\ell}$ contains at least $\ell$ vertices of $K_{\ell k_1+k_2-1}$. This is a contradiction. Hence $G_2(n, k_1, k_2, 2\ell)$ is $F(k_1, k_2; 2\ell)$-free and
\begin{equation}\label{eq2>=}
  ex(n, F(k_1, k_2; 2\ell))\geqslant e(G_2(n, k_1, k_2, 2\ell)).
\end{equation}

Now we prove Theorem \ref{mainthm2} by induction on $k_2$.  For $k_2=0$, $n\geqslant (2{\ell}^2+3\ell-4)k_1+3$, $G_2(n, k_1, 0, 2\ell)=K_{\ell k_1-1}\vee \overline{K}_{n-\ell k_1+1}$ and $F(k_1, 0; 2\ell)=k_1P_{2\ell}$ hold, and the results follow from Lemma \ref{ourlem}.
Suppose that $k_2\geqslant 1$ and Theorem \ref{mainthm2} holds for all $k_2'<k_2$. Suppose $G$ is an $F(k_1, k_2; 2\ell)$-free graph with $e (G)=ex(n, F(k_1, k_2; 2\ell))$. Hence, by (\ref{eq2>=}) and (\ref{2}), we have
  \begin{align*}
    e(G) & \geqslant e(G_2(n, k_1, k_2, 2\ell)) \\
     & =(\ell k_1+k_2-1)n-\frac{(\ell k_1+k_2)(\ell k_1+k_2-1)}{2} \\
     & >(\ell +k_2-2)n-\frac{(k_2-1)(k_2+2\ell-2)}{2}\\
     & = \frac{(k_2-1)(k_2-2)}{2}+(k_2+\ell-2)(n-k_2+1)\\
     & \geqslant ex(n,k_2S_{2\ell-1}),
  \end{align*}
which implies $G$ contains $k_2$ copies $S_{2\ell-1}$ by Theorem \ref{li2022} and Lemma \ref{sl}. By induction hypothesis,
$$ex(n-2\ell, F(k_1, k_2-1; 2\ell))=e(G_2(n-2\ell, k_1, k_2-1, 2\ell)).$$
Since $G$ is $F(k_1, k_2; 2\ell)$-free, $G-S_{2\ell-1}$ is $F(k_1, k_2-1; 2\ell)$-free. Hence,
\begin{equation}\label{eq2<=}
  e(G-S_{2\ell-1})\leqslant ex(n-2\ell, F(k_1, k_2-1; 2\ell))= e (G_2(n-2\ell, k_1, k_2-1, 2\ell)).
\end{equation}
Let $m_0$ be the number of edges incident with the vertices of $S_{2\ell-1}$ in $G$. Noting that $\ell\geqslant2$, $k_1\geqslant2$ and $n\geqslant (2{\ell}^2+3\ell-4)k_1+(4{\ell}^2-2\ell+1)k_2+3$, by (\ref{eq2>=}) and (\ref{eq2<=}), we have
\begin{align*}
    m_0 &= e(G)-e(G-S_{2\ell-1}) \\
   &\geqslant e(G_2(n, k_1, k_2, 2\ell))-e(G_2(n-2\ell, k_1, k_2-1, 2\ell))\\
   &= n+(2{\ell}^2-\ell)k_1+(2\ell-1)k_2-4\ell+1 \\
   &\geqslant 2\ell(2\ell(k_1+k_2)-1)+(2\ell-4)(k_1-1)\\
   &\geqslant 2\ell(2\ell(k_1+k_2)-1).
\end{align*}
  Then we may construct a vertex subset $U\subseteq V(G)$ of order $k_2$ whose each vertex has degree at least $2\ell(k_1+k_2)-1$. Write $\overline{U}=V(G)\backslash U$. Then $\left|\overline{U}\right|=n-k_2$. By (\ref{eq2>=}), we have
$$ e\left(G[\overline{U}]\right) =e(G)-e(G[U])-e\left(U, \overline{U}\right)
      \geqslant e(G_2(n, k_1, k_2, 2\ell))-e(G[U])-e\left(U, \overline{U}\right).$$
  We consider the following two cases.

\vspace{0.5cm}
  \textbf{Case 1.} $G[U]$ is a clique, and each vertex in $U$ is adjacent to each vertex in $\overline{U}$.

  In this case, $e(G[U])=k_2(k_2-1)/2$ and $e\left(U, \overline{U}\right)=k_2(n-k_2)$. Then by Lemma \ref{ourlem}, we have
  \begin{align*}
    e\left(G[\overline{U}]\right) & \geqslant e(G_2(n, k_1, k_2, 2\ell))-e(G[U])-e\left(U, \overline{U}\right) \\
     & = (\ell k_1-1)\left(n-\frac{\ell k_1}{2}-k_2\right)\\
     & = ex(n-k_2, k_1P_{2\ell}).
  \end{align*}
   If $e\left(G[\overline{U}]\right)>ex(n-k_2, k_1P_{2\ell})$, then we have $k_1P_{2\ell}\subseteq G[\overline{U}]$. Set $W=\overline{U}\backslash V(k_1P_{2\ell})$.
   Note that
   \begin{equation*}
     \left|W\right| =\left|\overline{U}\backslash V(k_1P_{2\ell}) \right| = n-k_2-2\ell k_1 \geqslant(2{\ell}^2+\ell-4)k_1+(4{\ell}^2-2\ell)k_2+3 >(2\ell-1)k_2,
   \end{equation*}
   and each vertex in $U$ is adjacent to each vertex in $W$. So there are $k_2$ copies of $S_{2\ell-1}$ in $G[V(G)\backslash V(k_1P_{2\ell})]$ with $k_2$ center vertices in $U$ and $(2\ell-1)k_2$ leaves vertices in $W$. Hence, we have $k_1P_{2\ell}\cup k_2S_{2\ell-1}\subseteq G$, which is a contradiction.

  Hence $e\left(G[\overline{U}]\right)=ex(n-k_2, k_1P_{2\ell})$ and $G[\overline{U}]$ does not contain $k_1$ copies of $P_{2\ell}$. By Lemma \ref{ourlem} again, $$G[\overline{U}]=\mbox{\rm{EX}}(n-k_2, k_1P_{2\ell})=K_{\ell k_1-1}\vee \overline{K}_{n-\ell k_1-k_2+1}.$$
  Hence, $$G=K_{k_2}\vee \left(K_{\ell k_1-1}\vee \overline{K}_{n-\ell k_1-k_2+1}\right)=G_2(n, k_1, k_2, 2\ell).$$

\vspace{0.5cm}
  \textbf{Case 2.} $G[U]$ is not a clique, or some vertex in $U$ is not adjacent to some vertex in $\overline{U}$.

  In this case, either $e(G[U])<k_2(k_2-1)/2$ or $e\left(U, \overline{U}\right)<k_2(n-k_2)$ holds. Then by Lemma \ref{ourlem}, we have
  \begin{align*}
   e\left(G[\overline{U}]\right) & \geqslant e(G_2(n, k_1, k_2, 2\ell))-e(G[U])-e\left(U, \overline{U}\right)\\
    & >(\ell k_1+k_2-1)n-\frac{(\ell k_1+k_2)(\ell k_1+k_2-1)}{2}-\frac{k_2(k_2-1)}{2}-k_2(n-k_2)\\
    & = (\ell k_1-1)\left(n-\frac{\ell k_1}{2}-k_2\right)\\
    & = ex(n-k_2, k_1P_{2\ell}),
  \end{align*}
  which implies $k_1P_{2\ell}\subseteq G[\overline{U}]$. Set $W=\overline{U}\backslash V(k_1P_{2\ell})$. For any vertex $u\in U$, $$d_{G[W]}(u)\geqslant (2\ell (k_1+k_2)-1)-(k_2-1)-2\ell k_1=(2\ell-1)k_2.$$
  Hence, we can find $k_2$ copies of $S_{2\ell-1}$ in $G[V(G)\backslash V(k_1P_{2\ell})]$ with $k_2$ center vertices in $U$ and $(2\ell-1)k_2$ leaves vertices in $W$. Hence, there is a copy of $k_1P_{2\ell}\cup k_2S_{2\ell-1}$ in $G$, which is a contradiction.
  The proof is completed.
  \end{proof}

\subsection{The Turán number and the extremal graph for $F(2, k; 5)$}

\begin{proof}[Proof of Theorem \ref{mainthm3}]
We suppose $n\geqslant 21k+38$ in this
subsection. Recall that $$G_3(n, k)=K_{k+3}\vee \left(K_2\cup \overline{K}_{n-k-5}\right)$$
and
$$F(2, k; 5)=2P_5\cup kS_4.$$

If $G_3(n, k)$ contains a copy of $F(2, k; 5)$, then each $S_4$ contains at least one vertex of $K_{k+3}$ and each $P_5$ contains at least two vertices of $K_{k+3}$. This is a contradiction. Hence $G_3(n, k)$ is $F(2, k; 5)$-free and
\begin{equation}\label{eq3>=}
  ex(n, F(2, k; 5))\geqslant e(G_3(n, k)).
\end{equation}

Now we prove Theorem \ref{mainthm3} by induction on $k$. For $k=0$, $n\geqslant38$, $G_3(n, 0)=K_3\vee\left(K_2\cup \overline{K}_{n-5}\right)$ and $F(2, 0; 5)=2P_5$ hold. Hence the results follow from Lemma \ref{2P5g}. Suppose that $k\geqslant 1$ and the results hold for all $k'<k$. Suppose $G$ is an $F(2, k; 5)$-free graph with $e(G)=ex(n, F(2, k; 5))$. Then by (\ref{eq3>=}) and (\ref{3}), we have
  \begin{align*}
    e(G) & \geqslant e(G_3\left(n,k\right)) \\
     & =(k+3)n-\frac{k^2+7k+10}{2} \\
     & > \left(k+\frac{1}{2}\right)n-\frac{k^2+2k-3}{2}\\
     & = \left(k-1\right)\left(n-\frac{k}{2}\right)+\frac{3(n-k+1)}{2}\\
     & \geqslant\left(k-1\right)\left(n-\frac{k}{2}\right)+\left\lfloor\frac{3(n-k+1)}{2}\right\rfloor\\
     & \geqslant ex(n,kS_4),
  \end{align*}
which implies $G$ contains $k$ copies $S_4$ by Theorem \ref{li2022} and Lemma \ref{sl}. By induction hypothesis,
$$ex(n-5, F(2, k-1; 5))=e(G_3(n-5, k-1)).$$
Since $G$ is $F(2, k; 5)$-free, $G-S_4$ is $F(2, k-1; 5)$-free. Hence,
 \begin{equation}\label{eq3<=}
   e(G-S_4)\leqslant ex(n-5, F(2, k-1; 5))= e(G_3(n-5, k-1)).
 \end{equation}
Let $m_0$ be the number of edges incident with the vertices of $S_4$ in $G$. Noting that $n\geqslant 21k+38$, by (\ref{eq3>=}) and (\ref{eq3<=}), we have
  \begin{align*}
    m_0 & =e(G)-e(G-S_4) \\
     & \geqslant e(G_3(n, k))-e(G_3(n-5, k-1)) \\
     & = n+4k+7 \\
     & \geqslant 5(5k+9).
  \end{align*}
  Then we can construct a vertex subset $U\subseteq V(G)$ of order $k$ whose each vertex has degree at least $5k+9$. Let $\overline{U}=V(G)\backslash U$. Then $\left|\overline{U}\right|=n-k$. Note that
$$ e\left(G[\overline{U}]\right)=e(G)-e(G[U])-e\left(U, \overline{U}\right)
     \geqslant e(G_3(n, k))-e(G[U])-e\left(U, \overline{U}\right).$$
  We consider the following two cases.

\vspace{0.5cm}
  \textbf{Case 1.} $G[U]$ is a clique and each vertex in $U$ is adjacent to each vertex in $\overline{U}$.

  In this case, $e(G[U])=k(k-1)/2$ and $e\left(U, \overline{U}\right)=k(n-k)$. Then by Lemma \ref{2P5g}, we have
  \begin{align*}
    e(G[\overline{U}]) & \geqslant e(G_3(n, k))-e(G[U])-e\left(U, \overline{U}]\right) \\
     & = 3(n-k)-5\\
     & = ex(n-k, 2P_5).
  \end{align*}
   If $e\left(G[\overline{U}]\right)>ex(n-k, 2P_5)$, then $2P_5\subseteq G[\overline{U}]$. Set $W=\overline{U}\backslash V(2P_5)$.
   Note that $$\left|W\right|=\left|\overline{U}\backslash V(2P_5)\right|=n-k-10\geqslant 21k+38-k-10>4k$$
   and each vertex in $U$ is adjacent to each vertex in $W$. Hence, there are $k$ copies of $S_4$ in $G[V(G)\backslash V(2P_5)]$ with $k$ center vertices in $U$ and $4k$ leaves vertices in $W$, and then $2P_5\cup kS_4\subseteq G$, which is a contradiction.

  Hence $e\left(G[\overline{U}]\right)=ex(n-k, 2P_5)$ and $G[\overline{U}]$ does not contain $2$ copies of $P_5$. By Lemma \ref{2P5g} again, $$G[\overline{U}]=\mbox{\rm{EX}}(n-k, 2P_5)=K_3\vee\left(K_2\cup \overline{K}_{n-k-5}\right),$$
  and then $$G=K_{k+3}\vee \left(K_2\cup \overline{K}_{n-k-5}\right)=G_3(n, k).$$

\vspace{0.5cm}
  \textbf{Case 2.} $G[U]$ is not a clique or some vertex in $U$ is not adjacent to some vertex in $\overline{U}$.

  In this case, either $e(G[U])<k(k-1)/2$ or $e\left(U, \overline{U}\right)<k(n-k)$ holds. Then by Lemma \ref{2P5g}, we have
  \begin{align*}
   e\left(G[\overline{U}]\right) & \geqslant e(G_3(n, k))-e(G[U])-e\left(U, \overline{U}\right)\\
    & > (k+3)n-\frac{k^2+7k+10}{2}-\frac{k(k-1)}{2}-k(n-k)\\
    & = 3(n-k)-5\\
    & = ex(n-k, 2P_5),
  \end{align*}
  which implies $2P_5\subseteq G[\overline{U}]$. Set $W=\overline{U}\backslash V(2P_5)$. For any vertex $u\in U$,
  $$d_{G[W]}(u)\geqslant (5k+9)-(k-1)-10=4k.$$
  Hence, we can find $k$ copies of $S_4$ in $G[V(G)\backslash V(2P_5)]$ with $k$ center vertices in $U$ and $4k$ leaves vertices in $W$. Hence, there is a copy of $2P_5\cup kS_4$ in $G$, which is a contradiction.
  The proof is completed.
\end{proof}

\end{document}